\documentclass[oneside,reqno]{amsart}
\usepackage{amssymb,amsmath}
\usepackage[english]{babel}
\usepackage{amsfonts}
\usepackage{color}
\usepackage{amsthm}
\usepackage{graphicx}
\usepackage{mathtools}
\usepackage{bm}
\usepackage{amsmath}
\usepackage{listings}
\usepackage{hyperref}
\usepackage{enumitem}

\usepackage{xy}
\usepackage[draft]{fixme}

\DeclareMathOperator{\Aut}{Aut}
\DeclareMathOperator{\Out}{Out}
\DeclareMathOperator{\Inn}{Inn}

\DeclareMathOperator{\conj}{\mathrm{conj}}

\newcommand{\id}{\mathrm{id}}
\newcommand{\Autc}{\Aut_{Col}}
\newcommand{\Outc}{\Out_{Col}}

\newcommand{\Span}[1]{\left\langle#1\right\rangle}

\let\phi\varphi
\theoremstyle{plain}
\newtheorem{theorem}{Theorem}[section]

\newtheorem{lemma}[theorem]{Lemma}
\newtheorem{proposition}[theorem]{Proposition}

\theoremstyle{remark}
\newtheorem{remark}{Remark}

\theoremstyle{definition}

\newtheorem{notation}[theorem]{Notation}

\usepackage{listings}
\lstdefinelanguage{GAP}{%
 morekeywords={%
 Assert,Info,IsBound,QUIT,%
 TryNextMethod,Unbind,and,break,%
 continue,do,elif,%
 else,end,false,fi,for,%
 function,if,in,local,%
 mod,not,od,or,%
 quit,rec,repeat,return,%
 then,true,until,while%
 },%
 sensitive,%
 morecomment=[l]\#,%
 morestring=[b]",%
 morestring=[b]',%
}[keywords,comments,strings]

\usepackage[T1]{fontenc}
\usepackage[variablett]{lmodern}
\usepackage{xcolor}
\begin{document}
\title[Coleman Automorphisms and Semidihedral Sylow 2-Subgroups]{Class-preserving Coleman Automorphisms of Finite Groups with Semidihedral Sylow 2-Subgroups} 
 \author[R.~Aragona]{Riccardo Aragona}

\address{DISIM \\
 Universit\`a degli Studi dell'Aquila\\
 via Vetoio\\
 I-67100 Coppito (AQ)\\
 Italy}       

\email{riccardo.aragona@univaq.it}

\date{} \thanks{The author's ORCiD: 0000-0001-8834-4358} 
\thanks{The author is member of INdAM-GNSAGA
 (Italy).}
 \thanks{The author declares that he has no conflict of interest.}

\subjclass[2010]{20D45, 20D20, 16U70} \keywords{Class-preserving automorphism, Coleman automorphism, Semidihedral 2-group, Sylow 2-subgroup, normalizer problem}

\begin{abstract}
In this paper, we prove that finite groups with semidihedral Sylow 2-subgroup  have Class-preserving Coleman outer automorphism group of odd order. As a consequence, these groups satisfy the \emph{normalizer problem}. In particular, we extend some existing results in the literature concerning class-preserving Coleman automorphisms of finite groups with semidihedral Sylow 2-subgroups.
\end{abstract}

\maketitle


\section{Introduction}

The study of class-preserving automorphisms of finite groups -- that is, automorphisms that preserve the conjugacy classes of the elements of the group -- has long been a subject of interest in group theory research. Since the early twentieth century, several researchers have investigated which finite groups admit only inner class-preserving automorphisms, or have sought to construct counterexamples to this property. For relevant references, see for example~\cite{bur,gross,her3,mazur,wall}. 

In this paper we focus on the study of class-preserving automorphisms that satisfy an additional property: they become inner when restricted to every Sylow \(p\)-subgroup. Automorphisms satisfying this property are called \emph{Coleman automorphisms}, due to their significance in the study of the normalizer of a finte group \(G\) in the unit group \(\mathcal{U}(\mathbb{Z}(G))\) of its integral group ring \(\mathbb{Z}G\), as shown by Coleman in~\cite[Theorem 1]{col}.

More precisely, let \(G\) be a finite group. We denote by \(\Aut_c(G)\) the subgroup of \(\Aut(G)\) consisting of all class-preserving automorphisms of \(G\); by \(\Autc(G)\), the subgroup of \(\Aut(G)\) consisting of all Coleman automorphisms of \(G\); and by \(\Aut_{\mathbb{Z}}(G)\), the subgroup of \(\Aut(G)\) consisting of all automorphisms of \(G\) that induce an inner automorphism on the integral group ring \(\mathbb{Z}G\). From Coleman's result~\cite[Theorem 1]{col}, together with a result of Jackowski and Marciniak~\cite[Proposition 2.3]{jacmar}, it follows that \(\Aut_{\mathbb{Z}}(G) \leq \Aut_c(G) \cap \Autc(G)\). Moreover, Krempa~\cite[Theorem 3.2]{jacmar} proved that \(\Aut_{\mathbb{Z}}(G)\) is an elementary abelian 2-group. Therefore, denoting by \(\Out_c(G) = \Aut_c(G)/\Inn(G)\) the class-preserving outer automorphism group and by \(\Outc(G) = \Autc(G)/\Inn(G)\) the Coleman outer automorphism group, we deduce that if \(G\) is a finite group such that 
\begin{equation}\label{property}
\Out_c(G) \cap \Outc(G) \text{ is of odd order,}
\end{equation}
then \(\Aut_{\mathbb{Z}}(G) = \Inn(G)\).

It is worth noting that classifying finite groups \(G\) for which \(\Aut_{\mathbb{Z}}(G) = \Inn(G)\) is equivalent to solving the well known \emph{normalizer problem}~\cite[Problem 43]{seh}, that is,
\[
N_{\mathcal{U}(\mathbb{Z}G)}(G) = G Z(\mathcal{U}(\mathbb{Z}G)).
\]


For further details on the normalizer problem, see, for example, \cite{MR1873381}.

In the literature, one can find several affirmative results (see, e.g.,~\cite{herkim,marrog,van,colsemidih}) and a counterexample provided by Hertweck in~\cite{her1} concerning the problem of identifying finite groups \(G\) that satisfy Property~\eqref{property}.
 
The aim of this paper is to prove that any finite group with semidihedral Sylow 2-subgroups satisfies Property~\eqref{property}, and consequently satisfies the normalizer problem. In particular, we extend~\cite[Theorem~3.5]{colsemidih}, where the authors prove that for any solvable finite group \(G\) with semidihedral Sylow 2-subgroups, \(\Out_c(G) \cap \Outc(G)\) has odd order. 
Moreover, considering a different definition of Coleman automorphisms given in~\cite{marrog} -- namely, using our notation, those automorphisms \(\phi \in \Aut_c \cap \Autc(G)\) such that \(\phi^2 \in \Inn(G)\) -- it follows that any such automorphism of odd order is inner. Hence, the main theorem of this paper also extends the results proved in~\cite{aragona}.

\section{Preliminary results}\label{sec2}
In this section, we state the main theorem of the paper and present some preliminary results from~\cite{her3,her4,herkim} that are useful for its proof.

Let $G$ be a finite group. Denote by $\Aut(G)$ the group of all automorphisms of $G$, and by $\Inn(G)$ the subgroup of $\Aut(G)$ consisting of all inner automorphisms. Define the subgroup $\Aut_c(G)\leq \Aut(G)$ as the set of all class-preserving automorphisms of $G$, i.e., those automorphisms that map each element of $G$ to a conjugate of itself. An automorphism $\phi \in \Aut(G)$ is called a \emph{Coleman automorphism} if, for every Sylow $p$-subgroup $S$ of $G$, we have $\phi|_S = \conj_g|_S$ for some $g \in G$, where $\conj_g$ denotes the inner automorphism defined by $\conj_g(x)=g^{-1}xg$. We denote by $\Autc(G)$ the group of Coleman automorphisms of $G$. Notice that both $\Aut_c(G)$ and $\Autc(G)$ contain $\Inn(G)$. Accordingly, we define $\Out_c(G) = \Aut_c(G)/\Inn(G)$ and $\Outc(G) = \Autc(G)/\Inn(G)$ as the class-preserving outer automorphism group and the Coleman outer automorphism group.


The aim of this paper is to prove the following result.
\begin{theorem}\label{main}
If $G$ is a finite group with semidihedral Sylow 2-subgroup, then \(G\) satisfies Property \eqref{property}; in other words, $\Out_c(G)\cap \Outc(G)$ is of odd order. 
\end{theorem}

In order to prove the main theorem, we will need the following results. As usual, we denote by $O_{p}(G)$ the maximal normal $p$-subgroup of $G$, which is the intersection of all  Sylow \(p\)-subgroups of \(G\), and by $O_{p'}(G)$ the maximal normal subgroup of $G$ whose order is coprime to $p$.

\begin{lemma}\label{lem1}
Let $N\unlhd G$ and let $p$ be a prime number which does not divide the order of $G/N$. Then the following hold:
\begin{itemize}
\item[(i)] If $\phi\in\Aut(G)$ is a class-preserving or a Coleman automorphism of $G$ of $p$-power order, then $\phi$ induces a class-preserving or a Coleman automorphism of $N$, respectively.
\item[(ii)] If $\Out_c(N)$ or $\Outc(N)$ is a $p'$-group, then so is $\Out_c(G)$ or $\Outc(G)$, respectively. Moreover, if $\Out_c(N)\cap \Outc(N)$ is a $p'$-group, then so is $\Out_c(G)\cap\Outc(G)$.
\end{itemize}
\end{lemma}
\begin{proof}
	See \cite[Corollary~5]{her3} and \cite[Corollary~3]{herkim}.
\end{proof}

\begin{proposition}\label{prop:normalizer}
Let $G$ be a semidirect product $K\rtimes H$, where $K$ is an elementary abelian $p$-group, for some prime $p$. Let $\phi\in\Aut(G)$ be defined by $h\phi=h$ for all $h\in H$, and $k\phi=k^m$ for all $k\in K$ and for some fixed positive integer $m$. If $\phi\in\Aut_c(G)$, then $\phi$ induces a class-preserving automorphism of $N_G(U)/U$, for each subgroup $U$ of index $p$   in $K$.
\end{proposition}
\begin{proof}
	Since $\Aut_c(G)=\Aut_{\mathbb{Q}}(G)$ (see \cite[Proposition~2.8]{jacmar}), setting $R=\mathbb{Q}$, the claim holds from \cite[Proposition~2.8]{her4}.
\end{proof}

\begin{lemma}{\cite[Lemma~5]{her3}}\label{lem:frattini}
	Let $p$ be a prime, and let $G$ be a finite group such that there is a non-inner automorphism of $G$ of $p$-power order which preserves the conjugacy classes of elements of $p$-power order, but induces inner automorphisms on all proper factor groups. Then the following hold.
	\begin{itemize}
		\item[(i)] The Frattini subgroup $\Phi(G)$ is a $p$-group.
		\item[(ii)] If Sylow \(p\)-subgroups of \(G\) are abelian, then \(O_p(G)=1\).
			\item[(iii)] If $O_p(G)=1$, and $C_G(N)\leq N$ for some \(1\ne N\unlhd G\), then \(\phi|_N\ne \gamma|_N\) for all \(\gamma\in\Inn(G)\).
	\end{itemize}
\end{lemma}

Finally, we will need the following well-known lemma, whose proof can be found in \cite[Lemma~2]{her3}.

\begin{lemma}\label{fixS}
	Let \(p\) be a prime, and an automorphism of \(G\) of \(p\)-power order. Assume further that there is \(N\unlhd G\) such that \(\phi\) fixes all elements of \(N\),and that \(\phi\) induces the identity on \(G/N\). Then \(\phi\) induces the identity on \(G/O_p(Z(N))\). If \(\phi\) fixes in addition a Sylow \(p\)-subgroup of \(G\) element-wise, then \(\phi\) is an inner automorphism.
\end{lemma}

Moreover, throughout the paper, the following remark will play an important role.

\begin{remark}{\cite[Remark~1]{her3}}\label{rem:modify}
	Let \(\phi\) be a non-inner (class-preserving) automorphism of \(G\), of order a power of \(p\). Assume that
	\begin{itemize}
		\item[(i)] there is \(U\leq G\) and \(\sigma\in\Inn(G)\) with \(\phi|_U=\sigma|_U\), or that
		\item[(ii)] there is \(N\unlhd G\) with \(N\phi=N\) such that \(\phi\) induces an inner automorphism on \(G/N\).
	\end{itemize}
	Then there is \(\gamma\in Inn(G)\) such that (i) \(\phi\gamma|_U= \id\) or (ii) \(\phi\gamma\) induces the identity on \(G/N\), and \(\phi\gamma\) is still a non-inner (class-preserving) automorphism of order a power of \(p\). Sometimes we will use this fact without any further comment, just saying that we ``modify by an inner automorphism'', when we intend to replace \(\phi\) by \(\phi\gamma\).
\end{remark}

\begin{notation}\label{notation}
From now on, we will consider $G$ a finite group with semidihedral Sylow 2-subgroups. Let us recall that a semidihedral group $S$ of order $2^n$ is a group with the following presentation
\[
SD_{2^n}=\langle a,b \mid a^{2^{n-1}}=b^2=1, bab=a^{2^{n-2}-1}\rangle ,
\]
or, in other words, is the semidirect product $C_{2^{n-1}}\rtimes C_2$, where $C_2$ acts on $C_{2^{n-1}}$ by $x\mapsto x^{2^{n-2}-1}$. In particular, all the elements in $S\setminus \langle a\rangle $ are \(a^{2i}b\), for \(0\leq i \leq 2^{n-2}-1\), of order  2, or \(a^{2i+1}b\), for \(0\leq i \leq 2^{n-2}-1\), of order 4. Moreover, we have the following useful properties:
\begin{itemize}
\item[(i)] $S$ possesses precisely three maximal subgroups, respectively the cyclic group $C=\Span{a}$, the generalized quaternion group $Q=\Span{a^2,ab}$, and the dihedral group $D=\Span{a^2,b}$. More generally, the subgroups of $S$ are cyclic, generalized quaternion, dihedral or Klein four-groups;
\item[(ii)]  the center of \(S\) is $Z(S)=\Span{a^{2^{n-2}}}$, which has order 2, and $S/Z(S)\cong D_{2^{n-2}}$, the dihedral group of order $2^{n-1}$;
\item[(iii)] the maximal cyclic subgroups of $C$, $D$ and $Q$ all concide with \(\Phi(S)=[S,S]=<a^2>\cong C_{2^{n-2}}\), and we have \(S/\Phi(S)=S/[S,S]\cong C_2 \times C_2\) ;
\item[(iv)] the  normal subgroups of \(S\) are the three maximal normal subgroups $C$, $D$ and $Q$, together with all the subgroups of $C$. In particular, every proper normal subgroup of \(S\) is either maximal or is cyclic and contained in \([S,S]\).
\end{itemize}

\noindent Notice that, by definition, $n\geq4$, or, in other words, $S$ has order at least $16$. For more details on the properties of a semidihedral group, see, e.g.,  \cite[Lemma~1, p. 9]{quasidih}.
\end{notation}

\section{Proof of  Theorem~\ref{main}}
In this section, we prove Theorem~\ref{main}, beginning with several lemmas needed to show that any finite group with a semidihedral Sylow 2-subgroup satisfies Property~\eqref{property}. 

Throughout this section, we assume that \(G\) is a minimal counterexample to Theorem~\ref{main}. In particular,  we suppose that \(G\) admits a class-preserving, non-inner Coleman automorphism  of order a power of 2, which we denote by \(\phi\), without further mention. Since a Sylow 2-subgroup of \(G\) is semidihedral, and its subgroups, together with the subgroup itself, are cyclic, dihedral, or generalized quaternion, then the Sylow 2-subgroups of subgroups and quotients of \(G\) are of these types. Hence, by the minimality of 
\(G\), \cite[Proposition~4.7]{her4}, and the main theorem of \cite{her2}, proper quotient, and proper subquotient of \(G\) satisfies Property \eqref{property}.

Let us denote by \(S\) a Sylow 2-subgroup of \(G\), and let \(a,y\in G\) be such that \(\Span{a}\) is the maximal cyclic subgroup of \(S\) not containing \(y)\). In other words, \(y\) has order 2 or 4.

\begin{remark}\label{Snonorm}
	Notice that any proper quotient \(G/N\) of \(G\) must have even order. Otherwise, \(N\) would contain a Sylow 2-subgroup \(S = SD_{2^n}\), and by Lemma~\ref{lem1}, since \(\phi\) is assumed to be a non-inner class-preserving Coleman automorphism of order a power of 2, \(\phi\) would induce a class-preserving Coleman automorphism on \(N\) of order a power of 2. This contradicts the minimality of \(G\). In particular, no proper normal subgroup of \(G\)  contains \(S\).
\end{remark}

Let \(F=F(G)\) denote the \emph{Fitting subgroup}  of \(G\), and let \(F^*=F^*(G)=FE\) be  the \emph{generalized Fitting subgroup} \(\) of \(G\), where \(E=E(G)\) is the \emph{layer} of \(G\). See \cite[Chapter~1, Section~4]{gorsimple} for more details on the properties of \(F\) and \(F^*\) that we will use in this paper. In particular, we will use that \(F=\prod_{p\in\pi(G)}O_p(G)\), \(F^*\) is a self centralizing subgroup and is a central product of \(F\) and \(E\), and 
\(Z(E)\) is an abelian normal subgroup of \(G\) contained in \(F\).

\medskip

The proof of Theorem~\ref{main} relies on several auxiliary lemmas.

\begin{lemma}\label{lemO2'}
	If \(G\) is a minimal counterexample to Theorem~\ref{main}, then \(O_{2'}(F)\) is non-trivial.
\end{lemma}
\begin{proof}
Suppose, by way of contradiction, that \(F\) is a 2-group; that is, \(F=O_2(G)\). By Remark~\ref{rem:modify}, we can modify \(\phi\) by an inner automorphism so that it fixes \(S\) element-wise.

If \(F=F^*\), then \(C_G(F)\leq F\). Since \(F\) is a 2-group, it follows that \(F\leq S\) and so \(F\) is fixed  element-wise by \(\phi\). Therefore, \(\phi\) induces the identity on \(G/F\), and hence, by Lemma~\ref{fixS}, \(\phi\in\Inn(G)\), which is a contradiction.

Hence, we can assume \(E\ne 1\). Suppose \(Z(E)=1\), it then follows that \(F^*=F\times E\). If \(E\) contains a cyclic Sylow 2-subgroup, then by the Burnside's normal \(p\)-complement theorem \cite[Theorem II, Section 243]{burnside}, it has a normal 2-complement, and hence, by Feit-Thompson theorem \cite{feitho}, it is solvable. This is impossible since \(E\) is a product of non-abelian simple groups. Therefore, \(E\) cannot contain a cyclic Sylow 2-subgroup. Since a Sylow 2-subgroup of \(F^*=F\times E\) is a subgroup of the semidihedral group \(S\), and a Sylow 2-subgroup of \(E\) cannot be cyclic, it follows that \(O_2(G)=F=1\), and \(F^*=E\) is a non-abelian simple group. In \cite{brauer,brauersuzuki,suzuki}, Brauer and Suzuki  proved that a generalized quaternion cannot be a Sylow 2-subgroup of a finite simple group. Moreover, by Remark~\ref{Snonorm}, a Sylow 2-subgroup of \(E\) cannot be a semidihedral group, since \(E\) is a normal subgroup of \(G\). Therefore a Sylow 2-subgroup of \(E\) can be only a dihedral group. 
In \cite{gorwal} it was proved that  the only finite simple groups with a dihedral Sylow 2-subgroup are \(\mathrm{PSL}_2(q)\), and \(\mathrm{A}_7\). Noting that every Coleman automorphism of a group is, up to an inner automorphism, a \(p\)-central automorphism for any prime \(p\) dividing the order of the group, it follows from \cite[Theorem~14]{herkim} that there exists a prime \(p\in\pi(E)\) such that all \(p\)-central automorphisms of \(E\) are inner. So, also in this case, we can modify \(\phi\) so that \(E=F^*\) is fixed element-wise. Consequently, \(\phi\) induces the identity on \(G/F^*\). Since \(Z(F^*)=Z(E)=1\), it follows that \(\phi=\id\), which is a contradiction. Thus we have \(1\ne Z(E)\leq F\) is a 2-group.

Let \(P\) be a Sylow 2-subgroup of \(E<F^*=FE\). As shown above,  \(P\) cannot be cyclic. Moreover, \(P\) cannot be  a Klein four-group, since it is a subgroup of a semidihedral group and \(F\) is a 2-group, hence it is non-abelian. Consequently, \(Z(E)=Z(P)\cong C_2\). On the  one hand, a Sylow 2-subgroup \(P/Z(P)\) of \(F^*/Z(E)\) is a dihedral group, since \(P\) is non-cyclic subgroup of a semidihedral group. On the other hand, since, by Remark~\ref{Snonorm}, a Sylow 2-subgroup of \(G\) cannot be a normal subgroup, we conclude that if \(F\ne 1\), then \(O_2(G)=F\) must be either a cyclic group or a dihedral group or a generalized quaternion group. We now show that $F^* = E$.
Recall that $F^* = FE$ with $[F, E] = 1$ and $F \cap E = Z(E)$.
Since $F = O_2(G)$ is a $2$-group, a Sylow $2$-subgroup of $F$ is $F$ itself.
Now, a Sylow $2$-subgroup of $F^*$ is of the form $F \cdot P$,
where $P$ is a Sylow $2$-subgroup of $E$ and $F \cap P = Z(E) \cong C_2$,
and this must be (isomorphic to) a subgroup of the semidihedral group $S$.
Since $[F, E] = 1$, we have $[F, P] = 1$, and so
$F \leq C_S(P)$. However, since $P$ is a non-cyclic subgroup of $S$
and $Z(S) \cong C_2$, it follows that $C_S(P) = Z(S) = Z(P)$.
Therefore $F \leq Z(P) = Z(E)$. Since we have already established
that $1 \neq Z(E) \leq F$, we conclude that $F = Z(E)$.
In particular, $F \leq E$, and hence $F^* = FE = E$.
As above, we can modify \(\phi\) so that it fixes \(F^*\) element-wise. Moreover, we can suppose that \(\phi\) stabilizes \(S\) and that \(P\leq S\). Hence, since \(\phi\) is a Coleman automorphism, there exists \(s\in S\) such that \(\phi|_S=\conj_s|_S\) and in particular \([P,s]=1\). If \(s\not\in P\), then \(\Span{P,s}\) is a non-abelian 2-group  with center of order at least 4, which is a contradiction. Otherwise, if \(s\in P\), then \(s\in Z(P)=Z(E)=Z(F^*)\). This implies that we can assume  \(\phi|_{F^*}=\id\) and also  \(\phi|_S=\id\), since \(s\in Z(P)=Z(S)\) and \(\phi|_S=\conj_s|_S\). By Lemma~\ref{fixS}, it follows that \(\phi\) is an inner automorphism, which is a contradiction. Therefore, \(O_{2'}(F)\ne 1\).
\end{proof}

\begin{lemma}\label{lem2}
	If \(G\) is a minimal counterexample to Theorem~\ref{main}, then the Frattini subgroup \(\Phi(G)\) of \(G\) is a 2-group. In particular, the subgroup \(O_{2'}(F)\) is a product of abelian minimal normal subgroup of \(G\) and it has a complement in \(G\).
\end{lemma}
\begin{proof}
	Note that \(G\) satisfies the hypotheses of Lemma~\ref{lem:frattini}, since it is a counterexample of Theorem~\ref{main}. Therefore \(\Phi(G)\) is a 2-group. By well-known properties of the Frattini subgroup and the Fitting subgroup (see, e.g. \cite[Chapter III, 3.14, 4.4, 4.5]{Hupp}),  since \(\Phi(G)\) is a 2-group, and so \(O_{2'}(G)\cap \Phi(G)=1\),  the second part of the claim also follows.
\end{proof}

\begin{lemma}\label{lem2bis}
	Let \(G\) be a minimal counterexample to Theorem~\ref{main}. If \(M\) is a minimal normal subgroup of \(G\) contained in \(O_{2'}(F)\), then \(M\) has a complement \(K\) in \(G\), and it is possible to modify \(\phi\) so that \(K\) is fixed  element-wise by \(\phi\). Moreover, if \(N \ne M\) is a normal subgroup of \(G\), then there exists a 2-element \(k \in K\) such that
	\begin{itemize}
		\item \(k\) inverts the elements of \(M\);
		\item the image of \(k\) in \(G/MN\) is a central involution; and
		\item \(\phi|_M = \conj_k|_M\).
	\end{itemize}
\end{lemma}
\begin{proof}
	
	By hypothesis, $M$ is a minimal normal subgroup contained in $O_{2'}(F)$.
	In particular, since $M$ is normal, it is a union of conjugacy classes,
	and so $\varphi(M) = M$; moreover, by Lemma~3.2, $O_{2'}(F)$ is a product
	of abelian minimal normal subgroups of $G$, and so, by
	\cite[Proposition~1]{her3}, $M$ has a complement $K$ such that
	$K\varphi = K$.
	
	Since $K \cong G/M$, the minimality of $G$ implies that $\varphi$
	induces an inner automorphism on $G/M$. Hence, there exists $h \in K$
	such that $\mathrm{conj}_{h}^{-1} \varphi $ induces the identity
	on $G/M$; that is, for every $l \in K$, we have
	$l(\mathrm{conj}_{h}^{-1}\varphi ) = lm$ for some $m \in M$.
	Since both $\varphi$ and $\mathrm{conj}_{h}$ preserve $K$
	(recall that $K\varphi = K$ and $h \in K$), the left-hand side
	belongs to $K$. As $G = M \rtimes K$ and $M \cap K = 1$, it follows
	that $m = 1$. Therefore, $\mathrm{conj}_{h}^{-1}  \varphi$ fixes
	$K$ element-wise. By Remark~1, it is then possible to modify $\varphi$
	by an inner automorphism so that $K$ is fixed element-wise by $\varphi$.
	
	Moreover, since $M$ is contained in a Sylow $p$-subgroup $Q$ of $G$
	and $\varphi$ is a Coleman automorphism, there exists $g \in G$ such that
	$\varphi|_Q = \mathrm{conj}_g|_Q$. Writing $g = mh$ with $m \in M$
	and $h \in K$, since $M$ is abelian and normal we obtain
	$\varphi|_M = \mathrm{conj}_h|_M$. Replacing $h$ by its $2$-part
	(which is possible since $\varphi$ has $2$-power order, and so
	$\mathrm{conj}_h|_M$ also has $2$-power order, implying that the
	odd-order part of $h$ acts trivially on $M$),
	we get a $2$-element $k \in K$ such that
	$\varphi|_M = \mathrm{conj}_k|_M$.
	Note that $k$ is determined by the Coleman property and does not depend
	on any choice of normal subgroup.
	
	Now, if $N \neq M$ is a normal subgroup of $G$, then by the minimality
	of $G$, $\varphi$ induces an inner automorphism on $G/N$; that is,
	$\varphi$ induces $\mathrm{conj}_{\bar{g}}$ on $G/N$ for some
	$\bar{g} \in G/N$. Since $\varphi$ has $2$-power order, also the
	induced automorphism has $2$-power order, so we may replace $\bar{g}$
	by its $2$-part and assume that $\bar{g}$ is a $2$-element of $G/N$.
	Now, $G/N = (MN/N) \rtimes (KN/N)$, and $MN/N$ has odd order (since
	$M$ has odd order). Therefore, all Sylow $2$-subgroups of $G/N$ are
	contained in conjugates of $KN/N$, and so every $2$-element of $G/N$
	is conjugate to an element of $KN/N$. Up to composing with an inner
	automorphism (which does not change the class modulo $\mathrm{Inn}(G/N)$),
	we may assume $\bar{g} = k_N N$ for some $2$-element $k_N \in K$.
	A~priori, $k_N$ depends on~$N$. However, since $M$ is a
	minimal normal subgroup and $N \neq M$, we have $N \cap M = 1$, and so
	$\mathrm{conj}_{k_N}|_M = \varphi|_M = \mathrm{conj}_k|_M$, which gives
	$k_N k^{-1} \in C_K(M)$. In particular, $k_N \equiv k \pmod{MN}$.
	Therefore, the image of $k$ in $G/MN$ is independent of the choice of~$N$,
	and $\varphi|_M = \mathrm{conj}_k|_M$.

	 Let $g = mh \in G$, with $m \in M$ and $h \in K$. Notice that, on the one hand,
	 \[
	 (gMN)\phi = (MgN)\phi = Mk^{-1} g k\, N = k^{-1} m h k\, MN,
	 \]
	 while, on the other hand,
	 \[
	 (gMN)\phi = (mh)\phi\, MN = k^{-1} m k\, h\, MN.
	 \]
	 Hence we obtain $h k\, MN = k h\, MN$.\\


	 Recalling that $MN = NM$, we have
	 \[
	 (MN k g)\phi = (NM k m h)\phi = NM k k^{-1} m k h = MN k h,
	 \]
	 and
	 \[
	 (MN g k)\phi = (NM m h k)\phi = (NM h k)\phi = MN h k.
	 \]
	 It follows that the image of $k$ is central in $G/MN$.\\
	 Moreover, note that $gMN = NMg = NMh = hMN$. We have shown that $[k, G] \leq MN$, and in particular the image
	 $\bar{k}$ of $k$ in $G/MN$ is central. We now show that $\bar{k}$
	 is an involution, i.e., $k^2 \in MN$.
	 Since $k \in S$ (up to conjugation, we may assume $k$ belongs to
	 a Sylow $2$-subgroup $S$ of $G$), and $MN \leq O_{2'}(F)$ has odd
	 order, we have $S \cap MN = 1$. Therefore, $S$ maps isomorphically
	 onto its image $\bar{S} = SMN/MN$ in $G/MN$, which is a Sylow
	 $2$-subgroup of $G/MN$. In particular, $\bar{S}$ is a semidihedral
	 group (isomorphic to $S$), and so $Z(\bar{S}) \cong C_2$.
	 Since $\bar{k} \in Z(\bar{S})$ (as $[k, G] \leq MN$ implies
	 $[\bar{k}, \bar{S}] = 1$), and $\bar{k}$ is a non-trivial $2$-element
	 (otherwise $k \in MN$, which is impossible since $k$ is a $2$-element
	 and $MN$ has odd order, so $k = 1$, contradicting the fact that $\varphi$
	 is non-inner), it follows that $\bar{k}$ generates $Z(\bar{S})$,
	 and in particular $\bar{k}^2 = 1$, that is, $k^2 \in MN$.
	 Therefore, the image of $k$ in $G/MN$ is a central involution. 
	 
	 
	Furthermore, \(C_M(k)=C_M(\phi)\unlhd G\), and so, since \(M\) is minimal normal subgroup of \(G\), we must have either \(C_M(\phi)=M\) or \(C_M(\phi)=1\). Since \(\phi\) does not act trivially on \(M\), otherwise it is the identity on both \(K\) and \(M\), i.e. on \(G\), we conclude that \(C_M(\phi)=1\). Finally, notice that \(\phi|_M=conj_k|_M\) is a fixed-point-free automorphism acting on th elementary abelian group \(M\), whose order is coprime to that of \(\phi\). Hence the action of  \(\phi|_M\) on the finite vector space $M$ is diagonalizable and, since it is fixed-point-free,  every eigenvalue needs to be -1. It follows that \(k\) acts by inversion on \(M\).
	\end{proof}
	
	\begin{lemma}\label{lem3}
		Let \(G\) be a minimal counterexample to Theorem~\ref{main}, and let \(M\) be a minimal normal subgroup of \(G\) contained in \(O_{2'}(F)\) with complement \(K\) in \(G\). Suppose that \(\phi\) fixes  \(K\) element-wise. Assume that \(M\) is not a cyclic group, and let \(U\) be a maximal subgroup of \(M\). Then \(\phi\) induces an inner automorphism on \(N_G(U)/U\).
	\end{lemma}
	
	\begin{proof}
	
			Note that $M$ is a subgroup of a Sylow $p$-subgroup $Q$ of~$G$.
			Since $\varphi$ is a Coleman automorphism, we have
			$\varphi|_Q = \operatorname{conj}_g|_Q$ for some $g \in G$.
			As $M$ is normal in~$G$, it follows that for each $m \in M \leq Q$
			we have $m\varphi = g^{-1}mg \in M$.
			Moreover, as shown above, $g$ acts on~$M$ by inversion.
			Therefore, since inversion is an automorphism of the abelian group~$M$
			and $U$ is a subgroup of~$M$, we have $g^{-1}Ug = U$;
			that is, $g$ normalizes~$U$.
			
			We now show that $\varphi$ induces a Coleman automorphism $\psi$
			on $N_G(U)/U$.
			First, observe that $\varphi$ stabilizes both $N_G(U)$ and~$U$:
			indeed, $\varphi$ fixes $K$ element-wise and acts on~$M$ by inversion
			(which preserves every subgroup of~$M$, in particular~$U$
			and its normalizer in~$G$).
			Hence $\varphi$ induces a well-defined automorphism
			$\psi$ of $N_G(U)/U$.
			
			Let $q$ be a prime and let $\bar{T}$ be a Sylow $q$-subgroup
			of $N_G(U)/U$.
			
			\medskip\noindent\emph{Case $q = p$.}
			Write $\bar{T} = TU/U$, where $T$ is a Sylow $p$-subgroup
			of $N_G(U)$.
			Then $T$ is contained in a Sylow $p$-subgroup $Q'$ of~$G$,
			and $\varphi|_{Q'} = \operatorname{conj}_{g'}|_{Q'}$
			for some $g' \in G$ which, by the same argument as above,
			acts on~$M$ by inversion and hence normalizes~$U$.
			In particular, $g' \in N_G(U)$, and for every $t \in T$ we have
			\[
		(tU)\psi = t\varphi\,U = g'^{-1}tg'\,U
			=(tU)\operatorname{conj}_{g'U}.
			\]
			Therefore $\psi|_{\bar{T}}$ coincides with the inner automorphism
			of $N_G(U)/U$ given by conjugation by $g'U$.
			
			\medskip\noindent\emph{Case $q \neq p$.}
			Since $G = M \rtimes K$ and $q \neq p$, every Sylow $q$-subgroup
			of $N_G(U)$ is conjugate in $N_G(U)$ to a subgroup of
			$N_K(U) = K \cap N_G(U)$.
			Because $\varphi$ fixes $K$ element-wise, it fixes $N_K(U)$
			element-wise, and therefore $\psi$ acts as the identity on every
			Sylow $q$-subgroup of $N_G(U)/U$ that is the image of a subgroup
			of~$N_K(U)$.
			In particular, $\psi|_{\bar{T}}$ is inner (being the identity).
			
			\medskip
			Since $\psi$ restricts to an inner automorphism on every Sylow
			subgroup of $N_G(U)/U$, we conclude that $\psi$ is a Coleman
			automorphism of $N_G(U)/U$.
			
			By Proposition~2.3, the automorphism $\psi$ is also
			class-preserving, and hence, by the minimality of~$G$,
			it must be inner.
	\end{proof}
	
	\begin{lemma}\label{lem4}
		Let \(G\) be a minimal counterexample to Theorem~\ref{main}. Then any proper cyclic quotient of \(G\) has order 2.
	\end{lemma}

\begin{proof}
	Since $G/N$ is cyclic with Sylow $2$-subgroup isomorphic to $C_2$, 
	we have $|G/N| = 2m$ for some positive odd integer~$m$. 
	Suppose, by way of contradiction, that $|G/N| > 2$, i.e., $m > 1$. 
	Let $L$ be the unique normal subgroup of $G$ containing $N$ such that 
	$|G/L| = m$. Since $m > 1$, $L$ is a proper normal subgroup of~$G$. 
	Moreover, since $|G:L| = m$ is odd, $L$ contains a Sylow $2$-subgroup 
	of~$G$, and hence the Sylow $2$-subgroups of $L$ are semidihedral. 
	By the minimality of~$G$, $L$ satisfies Property~\ref{property}. 
	Since $2 \nmid |G/L|$, Lemma~2.2(ii) implies that $G$ also satisfies 
	Property~\ref{property}, which contradicts the assumption that $G$ is a 
	counterexample. Therefore, $m = 1$ and $|G/N| = 2$.
\end{proof}


\begin{lemma}\label{lem5}
	Let \(G\) be a minimal counterexample to Theorem~\ref{main}, and let \(E\) be the layer of \(G\). Then \(O_2(G)E\) is non-trivial.
\end{lemma}
\begin{proof}
Suppose, by way of contradiction, that \(O_2(G)E = 1\). This implies that \(F^* = F\) is a \(2'\)-group and, by Lemma~\ref{lem2}, it is a direct product of abelian minimal normal subgroups \(M_i\) of \(G\), with \(1 \leq i \leq \ell\). Note that \(\ell > 1\). Indeed, if \(F = M_1\), since \(C_G(F) = C_G(F^*) \leq F^* = F\), by Lemma~\ref{lem:frattini}(iii), it follows that \(\phi|_{M_1} \ne \gamma|_{M_1}\) for each \(\gamma\in\Inn(G)\), which contradicts Lemma~\ref{lem2bis}. Throughout the remainder of the proof, I will repeatedly make use of the fact -- without further explicit mention -- that the number~$\ell$ of the $M_i$ is strictly greater than zero.

Since $O_2(G)E = 1$, we have $F = O_{2'}(F)$, which is a $2'$-group. 
By Lemma~\ref{lem2}, $F$ is a direct product of abelian minimal normal subgroups 
$M_1, \ldots, M_\ell$ of~$G$, and $F \cap \Phi(G) = 1$, so that $F$ has a 
complement~$H$ in~$G$, with $G = F \rtimes H$. Since $F \ne 1$ (by 
Lemma~\ref{lemO2'}), $G/F \cong H$ is a proper quotient of~$G$, and so, by the 
minimality of~$G$, $\varphi$ induces an inner automorphism on~$G/F$. By 
Remark~1, we can modify $\varphi$ by an inner automorphism so that it 
induces the identity on~$G/F$; equivalently, $\varphi$ fixes $H$ 
element-wise.

For each index~$i$, we now apply Lemma~\ref{lem2bis} to the minimal normal 
subgroup~$M_i$, with complement $K_i = \hat{M}_i \rtimes H$ and taking $N$ 
to be any normal subgroup of~$G$ different from~$M_i$ (for instance, 
$\hat{M}_i$). Since $\varphi$ already fixes $H$ element-wise and 
$H \leq K_i$, the $2$-element $h_i$ provided by Lemma~\ref{lem2bis} can be chosen 
in~$H$. In particular, $h_i$ inverts the elements of~$M_i$, its image 
in~$G/F \cong H$ is a central involution, and 
$\varphi|_{M_i} = \operatorname{conj}_{h_i}|_{M_i}$.
 
By Lemma~\ref{lem2bis}, applied with $M = M_i$ and $N = \hat{M}_i$ (so that 
$MN = F$), the involution $h_i$ maps to a central involution 
in $G/F \cong H$. Since $G = F \rtimes H$ and $F$ has odd order, 
we may choose the Sylow $2$-subgroup~$S$ to lie in~$H$. 
The projection $H \to G/F$ restricts to an isomorphism on~$H$, 
so $h_i$ is itself a central involution in~$H$. In particular, 
$h_i \in Z(H) \leq Z(S)$. Since $Z(S)$ has order~$2$ and contains 
a unique involution, namely $a^{2^{n-2}}$, it follows that 
$h_i = a^{2^{n-2}}$ for every~$i$. In particular, $h_i = h_j$ 
for all $1 \leq i, j \leq \ell$.


Now we fix an index \(i\), and assume that \(M_i\) is cyclic, generated by \(m_i\). Let \(K = C_H(M_i)\). Note that \(H \leq N_G(M_i)\) implies \(C_H(M_i) \unlhd H\). Moreover, since \(M_i\) is a cyclic minimal normal subgroup of odd order, it follows that \(M_i \cong C_p\) for some prime \(p\). Hence, \(H/K\) is isomorphic to a subgroup of \(C_{p-1}\), and so it is cyclic. Furthermore, \(h_i\) does not belong to \(K\), since it acts by inversion on \(M_i\), and so, by Lemma~\ref{lem4}, \(H/K \cong C_2\), generated by \(h_i\in H\). 
Thus, we conclude that \(H = K \rtimes \langle h_i \rangle\). Since $h_i$ maps to a central involution in $H/K$, we have 
$H = K \times \langle h_i \rangle$. Recall that, since $F$ has odd 
order, a Sylow $2$-subgroup~$S$ of~$G$ is contained in~$H$ (up to 
conjugacy). The projection $H \to \langle h_i \rangle \cong C_2$ 
restricts to a homomorphism $S \to C_2$, whose kernel $S \cap K$ 
has index at most~$2$ in~$S$. Since $S \cap K$ is a subgroup of~$S$ 
of index at most~$2$, it must contain one of the three maximal 
subgroups of~$S$ (the cyclic group~$C$, the dihedral group~$D$, 
or the generalized quaternion group~$Q$). In all cases, 
$h_i = a^{2^{n-2}} \in [S,S] \leq S \cap K \leq K$. 
But $h_i \in \langle h_i \rangle$ and $K \cap \langle h_i \rangle = 1$, 
which forces $h_i = 1$, a contradiction. Therefore, $M_i$ cannot 
be cyclic.


Now we modify \(\phi\) so that it fixes \(\hat{M}_1H\) element-wise and \(\phi|_{M_1} = \conj_{h_1}|_{M_1}\), and we choose a maximal subgroup \(U\) of \(M_1\). By Lemma~\ref{lem3}, \(\phi\) induces an inner automorphism on \(N_G(U)/U\). Then there exists a non-trivial 2-element \(h \in N_H(U)\) such that \([h, \hat{M}_1] = 1\) and \(h\) acts on \(M_1/U\) by inversion. Substituting \(S\) with a suitable conjugate, if necessary, we may suppose \(h \in S\). The same reasoning applied to another minimal normal subgroup, for instance \(M_2\), shows that there exists a non-trivial 2-element \(k \in H\) such that \([k, \hat{M}_{2}] = 1\),  so in particular \([k, M_1] = 1\), and we may suppose that \(k \in S\). Indeed, $k$ is a $2$-element, so it is contained in some Sylow 
$2$-subgroup~$S'$ of~$G$. Since all Sylow $2$-subgroups are conjugate, 
there exists $g \in G$ such that $S' = g^{-1}Sg$. Setting $k' = gkg^{-1} \in S$, 
we have $[k', M_1] = g[k, M_1]g^{-1} = 1$, since $[k, M_1] = 1$ and $M_1$ 
is normal in~$G$. Therefore, we may assume $k \in S$. As seen above, the element \(h_1\) in \(H\), acting on \(M_1\) by inversion, is a central involution in \(S\), and we may assume that it belongs to the maximal cyclic subgroup \(C = \langle a \rangle\) of \(S=SD_{2^{n}}\). Since \(h_1\) acts on \(M_1\) by inversion and \(M_1\) has no elements of order 2, it follows that \(h_1\) acts fixed-point free on \(M_1\). In particular, \(h_1 = a^{2^{n-2}}\). Moreover, \(a\) also acts fixed-point free on \(M_1\) and \([k, M_1] = 1\), and so we may suppose that \(k \not\in \Span{a}\), that is, \(k=a^sb\), for some \(s\in\mathbb{N}\) and \(b\) as  introduced in Notation~\ref{notation}. Also, \(h = a^t b\), for some \(t \in \mathbb{N}\), since \([h, \hat{M}_1] = 1\) and \(h_1\) -- which equals \(h_i\) for all \(i\) -- acts fixed-point free on \(\hat{M}_1\). Then \(h k^{-1} = a^{t - s}\). Since \(h = a^{t - s} k\), and \(k\) fixes \(M_1\) element-wise, it follows that \(h\) acts on \(M_1\) by conjugation with \(a^{t - s}\). Since \(h\) acts on \(M_1 / U\) by inversion, it follows that \(a^{t - s} = h_1\). Therefore, \(h = a^{t - s} k = h_1 k\), and this implies that \(\phi|_F = \conj_h|_F\), which contradicts Lemma~\ref{lem:frattini}(iii).
\end{proof}

\begin{lemma}\label{lem6}
Let \(G\) be an  minimal counterexample to Theorem~\ref{main}. Then \(O_2(G)\) is non-trivial.
\end{lemma}
\begin{proof}
	Suppose, by way of contradiction, that $O_2(G) = 1$. This implies 
	that the order of $F$ is odd, and by Lemma~\ref{lem5}, $E \neq 1$. 
	Moreover, $Z(E)$ also has odd order. However, since $Z(E) = \Phi(E) 
	\leq \Phi(G)$, and $\Phi(G)$ is a $2$-group, it follows that $Z(E) = 1$. 
	Since $Z(E) = 1$, the layer $E$ is a direct product of non-abelian 
	simple groups $L_1, \ldots, L_r$. Suppose $r \geq 2$. Then a Sylow 
	$2$-subgroup of~$E$ is of the form $T_1 \times \cdots \times T_r$, 
	where each~$T_i$ is a Sylow $2$-subgroup of~$L_i$. By Burnside's 
	normal $p$-complement theorem, a non-abelian simple group cannot have 
	a cyclic Sylow $2$-subgroup, and hence $|T_i| \geq 4$ for every~$i$. 
	In particular, $T_1 \times T_2$ is an abelian subgroup of~$S$ of 
	order at least~$16$ and rank at least~$2$. However, the only abelian 
	subgroup of~$S = SD_{2^n}$ of rank~$2$ is the Klein four-group, 
	which has order~$4$. This is a contradiction. Therefore $r = 1$, 
	and $E$ is a non-abelian simple group.  By Lemma~\ref{lem2}, there exists an abelian minimal normal subgroup of $O_{2'}(F)$ which has a complement $K$ in $G$. We claim that $E \leq K$. Since $M$ and $E$ are both normal in~$G$ 
	and $M \cap E \leq F \cap E = Z(E) = 1$, it follows that $[M, E] = 1$. 
	Now, consider the projection $\pi \colon G = M \rtimes K \to K$ 
	(with kernel~$M$). Since $E \cap M = 1$, the restriction $\pi|_E$ 
	is injective. Let us define $\sigma \colon E \to M$ such that \(e\sigma=e \cdot (e\pi)^{-1} \in M\) for each \(e\in E\). Notice that every $e \in E$ can be written as  $e = e\sigma \cdot e\pi$. A direct computation shows 
	that the map $\sigma \colon E \to M$ is a crossed homomorphism with 
	respect to the conjugation action of~$E\pi$ on~$M$. Since 
	$[E, M] = 1$, the subgroup $E\pi$ centralizes~$M$, and so $\sigma$ 
	is in fact a group homomorphism. As $E$ is perfect (being the layer 
	of~$G$) and $M$ is abelian, we have $E\sigma \leq [M, M] = 1$. 
	Therefore $\sigma = 1$, which means $e = e\pi \in K$ for all $e \in E$, 
	and hence $E \leq K$. Moreover, we can suppose that $\varphi$ fixes $K$ 
	element-wise and acts on $M$ by inversion.

	 Suppose that \(M\) is not cyclic. Let \(U\) be a maximal  subgroup of \(M\). By Lemma~\ref{lem3}, \(\phi\) induces an inner automorphism on \(N_G(U)/U\). Hence, there exists a 2-element \(g \in K\) such that \((h\phi)h^{-g} \in U\) for each \(h \in N_G(U)\). In particular, since \(E \leq K\) and \(\phi\) fixes \(K\) element-wise, we have \([E, g] = 1\), and \(g\) acts on \(M/U\) by inversion.
	Let \(P\) be a Sylow 2-subgroup of \(E\). As shown in the proof of Lemma~\ref{lemO2'}, \(P\) is not cyclic, \([P, g] = 1\), and \(\langle P, g \rangle\) is a 2-group with center of order greater than or equal to 4, which is impossible. Therefore, \(M\) is cyclic, and, by Lemma~\ref{lem4}, \(K/C_K(M)\cong C_2\).
	
Since $K/C_K(M) \cong C_2$ (as shown above by Lemma~\ref{lem4}), and 
$\bar{k} \notin C_K(M)/E$ (as $k$ acts on $M$ by inversion), 
we have $K/E = \langle \bar{k} \rangle \cdot (C_K(M)/E)$ with 
$\langle \bar{k} \rangle \cap (C_K(M)/E) = 1$. Moreover, 
by Lemma~\ref{lem2bis}, $\bar{k}$ is a central involution in $G/ME \cong K/E$, 
and so the product is direct: 
$K/E = \langle \bar{k} \rangle \times (C_K(M)/E)$.
	
	Assume that \(C_K(M) = E\), that is, \(G/ME = \Span{\bar{k}}\cong C_2\). Let \(1 \ne m \in M\) and \(e \in E\). Then there exists \(c \in K\) such that \((me)\phi = (me)^c\), that is, \(m^{-1} = m\phi = c^{-1}mc\) and \(e = e\phi = c^{-1}ec\). It follows that \(c = hk\), for some \(h \in E\), and so \(\phi|_E = \conj_k|_E\) is a class-preserving automorphism of \(E\). Hence, \(\conj_k|_E \in \Inn(E)\), considering the structure of \(E\) described in the proof of Lemma~\ref{lemO2'} and recalling that a class-preserving automorphism, up to an inner automorphism, is \(p\)-central for every prime \(p\) dividing the order of \(E\). In other words, there exists a 2-element \(e \in E\) such that \(\conj_k|_E = \conj_e|_E\), and, in particular, \([k, e] = 1\). This implies that \(ke^{-1}\) is a 2-element not belonging to \(E\), and it centralizes a Sylow 2-subgroup of \(E\). Note that this leads to a contradiction: indeed, \(E\) must be either a dihedral or a generalized quaternion group, and the subgroup generated by \(ke^{-1}\) centralizes a Sylow 2-subgroup of \(E\) while having trivial intersection with \(E\). However, this is not possible, since the Sylow 2-subgroups of \(G\) are semidihedral. Therefore, we conclude that \(C_K(M)/E \ne 1\).
	
	Let \(P\) be a Sylow 2-subgroup of \(C_K(M)/E\).  A Sylow $2$-subgroup of $K/E$ has the form 
	$\langle \bar{k} \rangle \times P \cong C_2 \times P$, 
	and is a quotient of the semidihedral group~$S$ 
	(since $F$ has odd order, a Sylow $2$-subgroup of~$G$ maps onto 
	a Sylow $2$-subgroup of $G/E \cong M \rtimes (K/E)$, and the 
	$2$-part concentrates in $K/E$). The only quotient of $S = SD_{2^n}$ 
	that admits a direct factor isomorphic to~$C_2$ is 
	$S/[S,S] \cong C_2 \times C_2$. Indeed, any proper quotient of~$S$ 
	is either cyclic or dihedral or isomorphic to $C_2 \times C_2$, 
	and among these, only $C_2 \times C_2$ contains a direct factor~$C_2$ 
	as a proper subgroup. It follows that the Sylow $2$-subgroup 
	of $K/E$ is $C_2 \times C_2$, and in particular $P \cong C_2$.
Hence, a Sylow 2-subgroup of \(G/E = MK/E\) is an elementary abelian 2-group of order 4, and so it is \(S/[S,S]\). It follows that a Sylow 2-subgroup of \(E\) is \([S,S]\), which is cyclic. However, \(E\) cannot contain a cyclic Sylow 2-subgroup, and so this leads to a contradiction. Therefore, we conclude that \(O_2(G) \ne 1\).
\end{proof}

\begin{lemma}\label{lem7}
	Let \(G\) be a minimal counterexample to Theorem~\ref{main}, and let \(S=SD_{2^{n}}\) be a semidihedral Sylow \(2\)-subgroup of \(G\). Then \(Z(S)\) is normal in \(G\). In particular, \(Z(S)=O_2(G)\)
\end{lemma}
\begin{proof}
	First note that, by Lemma~\ref{lem6} and Remark~\ref{Snonorm}, \(O_2(G)\) is a non-trivial proper normal subgroup of \(S\). 
	
	Suppose that \(O_2(G)\) is the maximal dihedral subgroup \(D\) of \(S\) or the maximal generalized quaternion subgroup \(Q\) of \(S\). Note that \(Z(S) = Z(D) = Z(Q)\), and since the center is characteristic in both \(D\) and \(Q\), it follows that \(Z(S)\) is normal in \(G\) in either case.
	
	Now suppose that \(O_2(G)\) is one of the cyclic subgroups contained in the maximal cyclic subgroup \(C\) of \(S\). If \(O_2(G) = Z(S)\), we are done. Otherwise, in each of the other cases -- recalling that every subgroup of a cyclic group is characteristic -- it follows again that \(Z(S)\) is normal in \(G\).
	
Finally, suppose \(O_2(G) \ne Z(S)\). It follows that \(O_2(G)\) has order at least 4. 

Let $M$ be a minimal normal subgroup of $G$ contained in $O_{2'}(F)$, 
whose existence is guaranteed by Lemma~\ref{lemO2'}, and let $K$ be a complement 
of $M$ in~$G$, as provided by Lemma~\ref{lem2bis}. We may assume that $\varphi$ fixes $K$ elementwise, and 
there exists a $2$-element $\tilde{k} \in K \cap S$ such that 
$\varphi|_M = \operatorname{conj}_{\tilde{k}}|_M$ and 
$[\tilde{k}, G] \leq MZ(S)$.

Now let $z \in O_2(G) \setminus Z(S)$ and $1 \neq m \in M$. 
Since $\varphi$ fixes $K$ elementwise and 
$\varphi|_M = \operatorname{conj}_{\tilde{k}}|_M$, we have 
$(mz)\varphi = m^{\tilde{k}} \cdot z$. On the other hand, 
$(mz)\varphi$ must be conjugate to $mz$ (as $\varphi$ is 
class-preserving), so there exists $g \in G$ such that 
$(mz)\varphi = g^{-1}(mz)g$. Writing $g = m'k$ with $m' \in M$ 
and $k \in K$, we note that $[M, O_2(G)] = 1$ (since $M \leq O_{2'}(F)$ 
and $O_2(G)$ are both normal in~$G$ with coprime orders, so 
$[M, O_2(G)] \leq M \cap O_2(G) = 1$). Together with the fact 
that $M$ is abelian, this gives $(mz)^{m'} = mz$, and hence 
$(mz)^g = (mz)^k$, so we may take $g = k \in K$. Since $\varphi$ 
has $2$-power order and acts on $M$ via $\tilde{k}$, we may further 
assume that $g$ is a $2$-element of~$K$.

Then there exists a 2-element \(g \in K\) such that \((mz)\phi = (mz)^g\), that is, \(m^{-1} = m\phi = m^g\) and \(z = z\phi = z^g\). It follows that \(\Span{O_2(G), g}\) is a non-cyclic abelian 2-group of order at least 8. We now prove the latter claim.

First, we show that $g$ centralizes $O_2(G)$.
From the equality $z = z^g$, we know that $g$ centralizes $z$.
Since we are in the case where $O_2(G)$ is a cyclic subgroup of the
maximal cyclic subgroup $C = \langle a \rangle$ of~$S$ (the cases
$O_2(G) \in \{D, Q\}$ having already been treated at the beginning of this proof), and $O_2(G)$
has order at least~$4$, the element $z \in O_2(G) \setminus Z(S)$
generates a subgroup $\langle z \rangle$ of index at most~$2$
in~$O_2(G)$. Now, $g$ is a $2$-element of~$S$ that centralizes
$\langle z \rangle$, and hence it induces an automorphism
of~$O_2(G)$ that fixes a subgroup of index at most~$2$ pointwise.
Since $O_2(G)$ is cyclic of $2$-power order, the only automorphism
of~$O_2(G)$ that fixes a subgroup of index at most~$2$ elementwise
is the identity. Therefore $g$ centralizes~$O_2(G)$, and so
$\langle O_2(G), g \rangle$ is abelian.

Next, we show that $\langle O_2(G), g \rangle$ is non-cyclic.
Suppose, for a contradiction, that $\langle O_2(G), g \rangle$ is
cyclic. Then $g \in O_2(G)$, since a cyclic $2$-group has a unique
subgroup of each order. But $O_2(G) \leq C \trianglelefteq S$ and
$[M, O_2(G)] = 1$, so $g \in O_2(G)$ would imply $m^g = m$ for
every $m \in M$, contradicting $m^g = m^{-1} \neq m$. Therefore
$g \notin O_2(G)$, and $\langle O_2(G), g \rangle$ is a non-cyclic
abelian $2$-group. Since $|O_2(G)| \geq 4$ and $g \notin O_2(G)$,
the group $\langle O_2(G), g \rangle$ has order at least~$8$.

Hence $\langle O_2(G), 
g \rangle$ is a non-cyclic abelian $2$-group of order at least $8$, 
which is impossible since all abelian subgroups of the semidihedral 
group $S$ have order at most $2^{n-1}$ and are cyclic or isomorphic 
to $C_2 \times C_2$. Hence, \(O_2(G) = Z(S)\).
\end{proof}

Before stating the following lemma, we recall some previously established properties of a minimal counterexample to Theorem~\ref{main}.
By Lemma~\ref{lemO2'} and Lemma~\ref{lem2bis}, there exists a minimal normal subgroup of \(G\) contained in \(O_{2'}(F)\), which has a complement \(K\) in \(G\). Modifying \(\phi\) according to Lemma~\ref{lem2bis}, we may assume that \(\phi\) fixes \(K\) element-wise. Furthermore, by Lemma~\ref{lem6} and again Lemma~\ref{lem2bis}, there exists a 2-element \(\tilde{k} \in K\) -- which may be taken to lie in \(S \leq K\) -- such that \([\tilde{k}, G] \leq MO_2(G)\), and both \(\phi\) and \(\tilde{k}\) act on \(M\) by inversion. Note that, by Lemma~\ref{lem3} and Lemma~\ref{lem7}, we also conclude that \([\tilde{k}, G] \leq MZ(S)\); in particular, \([\tilde{k}, S] \leq Z(S)\).

\begin{lemma}\label{lem8}
Let \(G\) be a minimal counterexample to Theorem~\ref{main}, and let \(M\) be a minimal normal subgroup of \(G\) contained in \(O_{2'}(F)\). Then \(M\) is cyclic.
\end{lemma}
\begin{proof}
The structure of the semidihedral group \(S = SD_{2^n}\) implies that if \(1 \ne [k, S] \leq Z(S)\), then \(\tilde{k} \in \Span{a}\) and has order 4; that is, \(\tilde{k} = a^{2^{n-3}}\). Let \(s = y\) be an element of order 2 or 4 in \(S \setminus C\), where \(C\) is the maximal normal cyclic subgroup of \(S\), and let \(P = \langle \tilde{k}, s \rangle\). Note that \(Z(P) = Z(S)\).

Assume, by way of contradiction, that \(M\) is not cyclic. Note that \(s^2 = 1\) or \(s^2 = a^{2^{n-2}} \in Z(S)\); hence \([s^2, M] = 1\), since, by Lemma~\ref{lem7}, \(Z(S)\cong C_2\) is normal in \(G\). In particular, there exists a maximal subgroup \(U\) of \(M\) such that \(P\) normalizes \(U\), that is, \(P \leq N_K(U)\). By Lemma~\ref{lem3}, \(\phi\) induces an inner automorphism on \(N_K(U)M/U\). Thus, there exists a $2$-element $g \in N_K(U)$ such that
$(h\varphi)h^{-g} \in U$ for all $h \in N_K(U)M$. Since $\varphi$
fixes $K$ element-wise and $P \leq N_K(U)$, we have $p\varphi = p$
for every $p \in P$. By Lemma~\ref{lem3}, $(p\varphi)p^{-g} \in U$ for
every $p \in P$, that is, $pp^{-g} \in U \leq M$. On the other
hand, since both $p$ and $g$ belong to~$K$, we have
$pp^{-g} \in K$. Therefore $pp^{-g} \in M \cap K = 1$, which gives
$[g, p] = 1$ for every $p \in P$. Hence $[g, P] = 1$, and so
$\langle P, g \rangle$ is a $2$-group.

If \(g \in P\), then \(g \in Z(P) = Z(S)\), and so \(g\) would act trivially on \(M\); however, this is impossible since, by Lemma~\ref{lem3}, \(g\) acts on \(M\) by inversion on \(M/U\). If \(g \notin P\), then \(\Span{P,g}\) would be a non-abelian 2-group with center of order at least 4, which is a contradiction, since \(Z(P) = Z(S) \cong C_2\). Therefore, we conclude that \(M\) is cyclic.
\end{proof}

Now we are ready to conclude the proof of Theorem~\ref{main}.

\begin{proof}[Proof of Theorem~\ref{main}]
Let us consider the notations introduced above, and let \(G\) be a minimal counterexample to Theorem~\ref{main}. By Lemma~\ref{lem8}, a minimal normal subgroup \(M\) of \(G\) contained in \(O_{2'}(F)\) is cyclic; in particular, \(M \cong C_p\) for some prime \(p\). It follows that \(K / C_K(M)\) is cyclic, and thus, by Lemma~\ref{lem4}, we have \(K / C_K(M) \cong C_2\).

By the beginning of the proof of Lemma~\ref{lem8}, we know that \(\tilde{k}\in K\) belongs to the cyclic subgroup of order 4 in \(S\), which is normal in \(S\) and lies in  \([S, S] \cong C_{2^{n-2}}\). This leads to a contradiction, since \(K/C_K(M) \cong C_2\), and so its Sylow 2-subgroups must be isomorphic to either \(S/C_{2^{n-1}}\), \(S/D_{2^{n-2}}\), or \(S/Q_{2^{n-1}}\). In all these cases, \([S, S] \leq C_K(M)\), implying that \(K\) acts trivially on \(M\). However, we know that \(\tilde{k}\in K\) acts on \(M\) by inversion. Hence, we obtain a contradiction.

This completes the proof of the main theorem.
\end{proof}

\section*{Acknowledgements}
The author  acknowledges  the funding support from MUR-Italy via PRIN 2022RFAZCJ "Algebraic methods in Cryptanalysis". The author wishes to express sincere gratitude to the referee for the careful reading of the manuscript and for the valuable suggestions, which have significantly contributed to improving the clarity of the exposition.

\bibliographystyle{acm}
\bibliography{biblio_coleman}

\begin{thebibliography}{10}

\bibitem{quasidih}
{\sc Alperin, J.~L., Brauer, R., and Gorenstein, D.}
\newblock Finite groups with quasi-dihedral and wreathed {S}ylow
  {$2$}-subgroups.
\newblock {\em Trans. Amer. Math. Soc. {\bf 151}\/} (1970), 1--261.

\bibitem{aragona}
{\sc Aragona, R.}
\newblock Coleman automorphisms of finite groups with semidihedral {S}ylow
  2-subgroups.
\newblock {\em Acta Math. Hungar. 172}, 2 (2024), 413--421.

\bibitem{brauer}
{\sc Brauer, R.}
\newblock Some applications of the theory of blocks of characters of finite
  groups. {II}.
\newblock {\em J. Algebra 1\/} (1964), 307--334.

\bibitem{brauersuzuki}
{\sc Brauer, R., and Suzuki, M.}
\newblock On finite groups of even order whose {$2$}-{S}ylow group is a
  quaternion group.
\newblock {\em Proc. Nat. Acad. Sci. U.S.A. 45\/} (1959), 1757--1759.

\bibitem{bur}
{\sc Burnside, W.}
\newblock On the {O}uter {I}somorphisms of a {G}roup.
\newblock {\em Proc. London Math. Soc. (2) {\bf 11}\/} (1913), 40--42.

\bibitem{burnside}
{\sc Burnside, W.}
\newblock {\em Theory of groups of finite order}.
\newblock Dover Publications, Inc., New York, 1955.
\newblock 2d ed.

\bibitem{col}
{\sc Coleman, D.~B.}
\newblock On the modular group ring of a {$p$}-group.
\newblock {\em Proc. Amer. Math. Soc. {\bf 15}\/} (1964), 511--514.

\bibitem{feitho}
{\sc Feit, W., and Thompson, J.~G.}
\newblock Solvability of groups of odd order.
\newblock {\em Pacific J. Math. {\bf 13}\/} (1963), 775--1029.

\bibitem{gorsimple}
{\sc Gorenstein, D., Lyons, R., and Solomon, R.}
\newblock {\em The classification of the finite simple groups}, vol.~40.1 of
  {\em Mathematical Surveys and Monographs}.
\newblock American Mathematical Society, Providence, RI, 1994.

\bibitem{gorwal}
{\sc Gorenstein, D., and Walter, J.~H.}
\newblock The characterization of finite groups with dihedral {S}ylow
  {$2$}-subgroups. {I}.
\newblock {\em J. Algebra {\bf 2}\/} (1965), 85--151.

\bibitem{gross}
{\sc Gross, F.}
\newblock Automorphisms which centralize a {S}ylow {$p$}-subgroup.
\newblock {\em J. Algebra {\bf 77}}, 1 (1982), 202--233.

\bibitem{her3}
{\sc Hertweck, M.}
\newblock Class-preserving automorphisms of finite groups.
\newblock {\em J. Algebra 241}, 1 (2001), 1--26.

\bibitem{her1}
{\sc Hertweck, M.}
\newblock A counterexample to the isomorphism problem for integral group rings.
\newblock {\em Ann. of Math. (2) {\bf 154 }}, 1 (2001), 115--138.

\bibitem{her4}
{\sc Hertweck, M.}
\newblock Local analysis of the normalizer problem.
\newblock {\em J. Pure Appl. Algebra 163}, 3 (2001), 259--276.

\bibitem{her2}
{\sc Hertweck, M.}
\newblock Class-preserving {C}oleman automorphisms of finite groups.
\newblock {\em Monatsh. Math. {\bf 136}}, 1 (2002), 1--7.

\bibitem{herkim}
{\sc Hertweck, M., and Kimmerle, W.}
\newblock Coleman automorphisms of finite groups.
\newblock {\em Math. Z. {\bf 242}}, 2 (2002), 203--215.

\bibitem{Hupp}
{\sc Huppert, B.}
\newblock {\em Endliche {G}ruppen. {I}}, vol.~Band 134 of {\em Die Grundlehren
  der mathematischen Wissenschaften}.
\newblock Springer-Verlag, Berlin-New York, 1967.

\bibitem{jacmar}
{\sc Jackowski, S., and Marciniak, Z.}
\newblock Group automorphisms inducing the identity map on cohomology.
\newblock In {\em Proceedings of the {N}orthwestern conference on cohomology of
  groups ({E}vanston, {I}ll., 1985)\/} (1987), vol.~{\bf 44}, pp.~241--250.

\bibitem{MR1873381}
{\sc Jespers, E., Juriaans, S.~O., de~Miranda, J.~M., and Rogerio, J.~R.}
\newblock On the normalizer problem.
\newblock {\em J. Algebra 247}, 1 (2002), 24--36.

\bibitem{marrog}
{\sc Marciniak, Z.~S., and Roggenkamp, K.~W.}
\newblock The normalizer of a finite group in its integral group ring and
  \v{C}ech cohomology.
\newblock In {\em Algebra---representation theory ({C}onstanta, 2000)},
  vol.~{\bf 28 } of {\em NATO Sci. Ser. II Math. Phys. Chem.} Kluwer Acad.
  Publ., Dordrecht, 2001, pp.~159--188.

\bibitem{mazur}
{\sc Mazur, M.}
\newblock Automorphisms of finite groups.
\newblock {\em Comm. Algebra {\bf 22}}, 15 (1994), 6259--6271.

\bibitem{seh}
{\sc Sehgal, S.~K.}
\newblock {\em Units in integral group rings}, vol.~69 of {\em Pitman
  Monographs and Surveys in Pure and Applied Mathematics}.
\newblock Longman Scientific \& Technical, Harlow, 1993.

\bibitem{suzuki}
{\sc Suzuki, M.}
\newblock Applications of group characters.
\newblock In {\em Proc. {S}ympos. {P}ure {M}ath., {V}ol. {VI}}. Amer. Math.
  Soc., Providence, RI, 1962, pp.~101--105.

\bibitem{van}
{\sc Van~Antwerpen, A.}
\newblock Coleman automorphisms of finite groups and their minimal normal
  subgroups.
\newblock {\em J. Pure Appl. Algebra 222}, 11 (2018), 3379--3394.

\bibitem{wall}
{\sc Wall, G.~E.}
\newblock Finite groups with class-preserving outer automorphisms.
\newblock {\em J. London Math. Soc. {\bf 22}\/} (1947), 315--320 (1948).

\bibitem{colsemidih}
{\sc Zheng, T., and Guo, X.}
\newblock Class-preserving {C}oleman automorphisms of finite groups whose
  {S}ylow 2-subgroups are semidihedral.
\newblock {\em J. Algebra Appl. {\bf 21}}, 8 (2022), Paper No. 2250166, 14.

\end{thebibliography}

\end{document}